\documentclass[11pt]{amsart}
\usepackage[utf8]{inputenc}

\usepackage{amsfonts}
\usepackage{amsthm}
\usepackage{perpage}
\MakeSorted{figure}
\MakeSorted{table}
\usepackage{latexsym}
\usepackage[dvips]{graphicx}
\usepackage{color}
\usepackage{amsmath,amssymb,graphics,amsthm}
\usepackage[english]{babel}

\usepackage{mathrsfs}
\usepackage{graphicx}
\usepackage[colorinlistoftodos]{todonotes}
\usepackage{enumerate} 
\usepackage{hyperref}
\usepackage{fullpage}

\newtheorem{thm}{Theorem}[section]

\newtheorem{lema}[thm]{Lemma}
\newtheorem{rem}[thm]{Remark}
\newtheorem{prop}[thm]{Proposition}

\newtheorem{ex}[thm]{Example}
\newtheorem{question}[thm]{Question}

\newcommand{\N}{{\mathbb N}}

\newcommand{\dom}{\text{\sf dom}}
\newcommand{\im}{\text{\sf im}}
\newcommand{\ev}{\text{\sf ev}}

\newcommand{\cantor}{2^{\N}}

\def\base{\mathcal{B}}

\def\hco{\tau_{hco}}
\def\DD{{\sf D}}
\def\II{{\sf I}}

\title{On the Polishness of the inverse semigroup $\Gamma(X)$ on a compact metric space $X$.}
\author{Jerson P\'erez}
\address{Escuela de Matem\'aticas, Universidad Industrial de Santander, C.P. 680001, Bucaramanga - Colombia.}

\email{jersonenrique\_64@hotmail.com}

\author{Carlos Uzc\'{a}tegui}
\address{Escuela de Matem\'aticas, Universidad Industrial de Santander, C.P. 680001, Bucaramanga - Colombia.}
\email{cuzcatea@saber.uis.edu.co}

\date{}


\begin{document}

\maketitle
\begin{abstract}
 Let $\Gamma(X)$ be the inverse semigroup of partial homeomorphisms between open subsets of a compact metric space $X$. There is a  topology, denoted $\tau_{hco}$, that makes $\Gamma(X)$ a topological inverse semigroup. We address the question of whether $\tau_{hco}$ is Polish. For a 0-dimensional compact metric space $X$, we prove that $(\Gamma(X), \tau_{hco})$  is Polish by showing that it is topologically isomorphic to a closed subsemigroup of the Polish symmetric inverse semigroup $I(\N)$. We present examples, similar to the classical Munn semigroups,  of Polish inverse semigroups consisting of partial isomorphism on lattices of open sets. 
\end{abstract}

\section{Introduction}

An inverse semigroup is a semigroup $S$ where for each $s\in S$  there is a unique $t\in S$ such that $sts=s$ and $tst=t$. 
The collection $\Gamma(X)$  of all homeomorphisms between open subsets of a topological space $X$ is the archetype of all inverse semigroups (see the introduction of Lawson's book \cite{Law}). When $X$ has the discrete topology we get the symmetric inverse semigroup $I(X)$ of all bijections between arbitrary subsets of $X$ (see, for instance, \cite[page 6]{Law}). This semigroup plays an universal role among inverse semigroups as the symmetric group does for groups. Namely, a classical theorem of Wagner and Preston says that every inverse semigroup is isomorphic to a subsemigroup of $I(X)$ for some set $X$.

Recently, $\Gamma(X)$  has been  endowed with a Hausdorff inverse semigroup topology $\tau_{hco}$ when $X$ is a locally compact Hausdorff space \cite{MPU}.  The topology $\tau_{hco}$ is a natural generalization of the classical compact-open topology on the group of homeomorphisms $H(X)$. When $X$ is discrete, this topology $\tau_{hco}$ was independently defined in   \cite{elliott2020} and \cite{PerezUzca2022} and will be denoted by $\tau_{pp}$.  An important particular case is when $X$ is countable, since  $(I(\N), \tau_{pp})$ is a Polish  (i.e. completely metrizable and separable) inverse semigroup. One of the main results of this paper is  that $(\Gamma(X), \tau_{hco})$ is Polish for $X$ a 0-dimensional compact metric space.  The key fact is that $\Gamma(X)$ turns out to be  topologically isomorphic to a  closed inverse subsemigroup of $(I(\N),\tau_{pp})$, as we explain below.

Motivated by the  classical Munn semigroups,  we study semigroups of partial lattice isomorphisms. 
Let $\base$ be a countable base of open subsets of a compact metric space $X$, which is assumed to be  closed under finite unions and intersections and contains $\emptyset$. Let $S(\base)$ be the inverse subsemigroup of $I(\base)$ consisting of all order isomorphisms (with respect to $\subseteq$)  between hereditary sublattices of $\base$. We prove that $S(\base)$ is a closed inverse subsemigroup of $(I(\base),\tau_{pp})$, and thus it is Polish.  

When $X$ is  a 0-dimensional compact metric space,   we show that $(\Gamma(X), \tau_{hco})$ is topologically isomorphic to  $(S(\base),\tau_{pp})$,  where $\base$ is any base of $X$ consisting of clopen sets. Thus, in this case,  $\Gamma(X)$ is  Polish. We do not know if the same holds for any compact metric space. This result can be regarded as an inverse semigroup version of the following known result about the group of homeomorphisms $H(X)$. Let $X$ be a compact 0-dimensional metric space.  Then $H(X)$ (with the compact-open topology)  is topologically isomorphic to a closed subgroup of the symmetric group $S_\infty(\N)$ of all permutation of $\N$ with the product topology (see \cite[Thm 1.5.1]{BeckerKechris1996}).

\section{Preliminaries}

A topological space is Polish if it is completely metrizable and separable. We refer to \cite{kechris1995} as a general reference for the descriptive set theory of Polish spaces.   
$\N$ denotes the non negative integer.  For any countable set $M$, we identify a subset $A\subseteq M$ with its characteristic function and thus a collection $\mathcal{C}$ of subsets of $M$ is seen as a subset of $\{0,1\}^M$ (which is homeormorphic to the Cantor space $\cantor$) and we can talk about closed, open, $F_\sigma$, $G_\delta$ etc. collections of subsets of $M$.

A {\em semigroup} is a non-empty set $S$ together with an associative binary operation. 
A semigroup  $S$ is  {\em regular},  if for all $s\in S$, there is $t\in S$ such that $st s=s$ and $tst=t$. In this case, $t$ is called  an inverse of $s$. If each element have a unique inverse,  $S$ is an {\em inverse} semigroup and, in this case, $s^{-1}$ denotes  the inverse of $s$. An element $s$ of a semigroup is {\em idempotent} if $ss=s$. We denote by $E(S)$ the collection of idempotents of $S$. We use \cite{Howie, Law} as general references for semigroup theory

Let $\tau$ be a topology  on a semigroup $S$.  If the multiplication $S\times S\to S$ is continuous, then  $S$ is called a {\em topological semigroup}.  An inverse semigroup $S$ is called topological, if it is a topological semigroup and the function $i:S\to S$, $s\to s^{-1}$ is continuous.

The symmetric inverse semigroup on a set $X$ is defined as follows:
\[
I(X) = \{f : A \rightarrow B |\;A, B \subseteq X\; \mbox{and $f$ is bijective}\}.
\]
For $f : A \rightarrow B$ in $I(X)$ we denote $A = \dom(f)$  and $B = \im(f)$. The operation on $I(X)$ is the usual composition, namely, given $f, g \in I(X)$, then $f \circ g$ is defined by letting $\dom(f \circ g)=g^{-1}(\dom(f) \cap \im(g))$ and $(f \circ g)(x) = f(g(x))$, if $x \in \dom(f \circ g)$. The idempotents  of $I(X)$ are the partial identities $1_{A} : A \rightarrow A, 1_{A}(x) = x$ for all $x \in A$ and $A \subseteq X$.
Notice that $1_{\emptyset}$ is the empty function which also belongs to $I(X)$.
$S_{\infty}(X)$ denotes  the symmetric group, that is, the collection of all bijections from $X$ to $X$.  To avoid a possible confusion we denote by $f[U]$ the direct image of $U$ under $f$ for $U\subseteq \dom(f)$.

The following functions play an analogous role in $I(X)$ as the projection functions do in  the product space $X^X$. Let $D_{x} = \{f\in I(X):\; x\in 
\dom(f)\}$ and $2^{X}$ denotes the power set of $X$.
\[
\begin{array}{lll}
\dom : I(X) \rightarrow 2^{X}, &f \mapsto  \dom(f),\\
\im : I(X) \rightarrow 2^{X}, & f \mapsto  \im(f),\\
\ev_{x} : D_x \rightarrow X, & f \mapsto  f(x),\;\mbox{for $x \in X$}.
\end{array}
\]
For $x, y \in X$, let
\[
\begin{array}{lcl}
v(x, y) & = &  \{f \in I(X)\; |\; x \in \dom(f) \;\mbox{and}\; f(x) = y\},\\
w_{1}(x) & = & \{f \in I(X)\; |\; x \notin \dom(f)\},\\
w_{2}(y) & = & \{f \in I(X)\; |\; y \notin \im(f)\}.
\end{array}
\]
The topology generated by  the sets $v(x,y)$, $w_1(x)$ and $w_2(x)$ is denoted by $\tau_{pp}$, it was defined in \cite{elliott2020,PerezUzca2022}. The topology $\tau_{pp}$ is the smallest Hausdorff inverse semigroup topology that makes continuous the functions $\dom$, $\im$ and $\ev_x$ (for $x\in X$), where $2^X$ is given the product topology.  

A special case for our purposes is  $I(\N)$.  It is known that  $(I(\N), \tau_{pp})$ is a Polish inverse semigroup \cite{elliott2020,PerezUzca2022}. Clearly, we can use any countable set $\base$ in place of $\N$ and work in the space $I(\base)$. Convergence in $I(\N)$ is as follows.

\begin{prop}(\cite{PerezUzca2022})
\label{conv}
Let $(f_n)_n$  be a sequence in $I(\N)$ and $f\in I(\N)$. Then, $f_n\to f$ if and only if  the following conditions hold.

\begin{itemize}
\item[(i)] For all $x\in \dom(f)$ there is $n_{0}\in \N$ such that  $x\in \dom(f_n)$  and $f_n(x)=f(x)$ for all  $n\geq n_0$.

\item[(ii)]  For all $x\not\in \dom(f)$ there is $n_0\in\N$ such that $x\not\in \dom(f_n)$  for all $n\geq n_0$.
\end{itemize}
\end{prop}

Let $X$ be a locally compact Hausdorff space.  We denote by  $\Gamma(X)$ the collection of all homeomorphism $f:U\to V$  where  $U$ and $V$ are  open subsets of  $X$.   Then $\Gamma(X)$ is an inverse subsemigroup of  $I(X)$. 

We recall the basic open sets in $\Gamma(X)$. For each compact set $K$ and open set in $X$ we let 
\[
    \langle K;V\rangle=\{f\in \Gamma(X): K\subseteq\dom(f) \;\&\; f[K]\subseteq V\}.
\]
\[
\langle K;V\rangle^{-1}=\{f\in \Gamma(X): K\subseteq f[V] \}.
\]
Now we define the Fell topology on  the collection $CL(X)$ of closed subsets of $X$ (including the emptyset). For each open set $V\subseteq X$, consider the following subsets of $CL(X)$:
\[
V^-=\{A\in CL(X):\; A\cap V\neq \emptyset\}
\]
and
\[
V^+=\{A\in CL(X):\; A\subseteq V\}.
\]
The {\em Fell} topology on $CL(X)$ has as a subbasis  the sets $V^-$ for $V$ open and $W^+$ for $W$ open with compact complement.
It is known that $CL(X)$ with the Fell topology is a compact Hausdorff space (see, for instance, \cite[Theorem 5.3.1]{Beer1993}).

Let $\tau_{co}$ be the topology generated by the sets $\langle K;V\rangle$ and $\langle K;V\rangle^{-1}$ for all compact $K$ and open $V$. Then $(\Gamma(X), \tau_{co})$ is a topological inverse semigroup. However, $\tau_{co}$ is not Hausdorff.  To remedy this,  consider the functions ${\sf D}, {\sf I}:\Gamma(X)\to CL(X)$ given by ${\sf D}(f)=X\setminus \dom(f)$ and ${\sf I}(f)=X\setminus \im(f)$.  
Let  $\hco$ be the  smallest topology extending $\tau_{co}$ and such that  $\sf D$ and $\sf I$ are continuous when $CL(X)$ is endowed with the Fell topology. Notice that $\II(f)=\DD(f^{-1})$ and ${\DD}^{-1}(V^+)=\langle V^c; X\rangle$. From this it follows that  the new open sets added to $\tau_{co}$ are   $\DD^{-1}(V^-)$ for $V\subseteq X$ open. 
 
\begin{thm} (\cite[Thm 3.11]{MPU}). 
Let $X$ be  a locally compact Hausdorff space. Then,  $(\Gamma(X),\tau_{hco})$ is a Hausdorff topological inverse semigroup  
\end{thm}

\begin{thm}
Let $X$ be a compact metric space. Then $(\Gamma(X), \hco)$ is a separable  metrizable space. 
\end{thm}

The proof follows from the next two facts and the Urysohn's metrization theorem.

\begin{lema}
Let $X$ be a compact metric space. Then $(\Gamma(X),\hco)$ is second countable. 
\end{lema}

\begin{proof}
Since $X$ is compact metric, we can fix a countable basis  $\base$ for $X$ closed under finite unions. Then the following is a basis for $\hco$:
\[
\{\langle \overline{U}; V\rangle:\,U,V\in \base\}\,\cup\, \{\langle \overline{U}; V\rangle^{-1}:\,U,V\in \base\}\,\cup\,\{\DD(V^-): \,V\in \base\}.
\]
\end{proof}

\begin{lema}
Let $X$ be a compact metric space. Then $(\Gamma(X),\hco)$ is regular. 
\end{lema}

\begin{proof}
Let $f\in \langle K; V\rangle$ with $K$ compact and $V$ open. We show that there is $M$ open such that $f\in M\subseteq \overline{M}\subseteq \langle K; V\rangle$. Let $U$ be open such that $K\subseteq U\subseteq \overline{U}\subseteq \dom(f)$. Let $W$ be open such that $f[K]\subseteq W\subseteq \overline{W}\subseteq V$. Consider $M=\langle K; W\rangle\cap \langle \overline{U}; X\rangle$. Notice that $M$ is open since $\DD^{-1}(X\setminus \overline{U})^+=\langle \overline{U}; X\rangle$. Clearly $f\in M$. 

We show that $\overline{M}\subseteq \langle K; \overline{W}\rangle$. Since $\langle K; \overline{W}\rangle \subseteq \langle K; V\rangle $ we are done. 
Let $g\in \overline{M}$.  We first show that $K\subseteq \dom(g)$.
In fact, suppose not and let $x\in K\setminus \dom(g)$. As $K\subseteq U$,   $g\in \DD^{-1}(U^-)$. This contradicts the fact that  $M\cap \DD^{-1}(U^-)=\emptyset$.
Finally, we show that $g\in \langle K; \overline{W}\rangle$. Suppose not and let $x\in K$ be such that $g(x)\not\in\overline{W} $. Then $g\in \langle \{x\}; X\setminus \overline{W}\rangle $. This contradicts the fact that $\langle \{x\}; X\setminus \overline{W}\rangle\cap M=\emptyset$.
\end{proof}

For the rest of the paper, we will always use on $\Gamma(X)$ the topology $\tau_{hco}$.  The next result will be needed in the sequel. 

\begin{prop}
\label{equality}
Let $X$ be a compact metric space and $V,W$ open sets. Then 
$E(V;W)=\{f\in \Gamma(X):\,V\subseteq \dom(f)\;\& \;W\subseteq \im(f)\;\&  \;f[V]=W\}$ is $G_\delta$.
\end{prop}
\proof
Let $K_n$ and $L_n$ be compact sets such that $V=\bigcup_n K_n$ and $W=\bigcup_n L_n$. Then 
\[
E(V;W)=\bigcap_n \;\langle K_n, W\rangle\; \cap \;\langle L_n, V\rangle^{-1}.
\]
\endproof

\section{Semigroups of partial lattice isomorphisms}

In this section we explore semigroups consisting of lattice isomorphisms. We start recalling the classical Munn semigroups following   \cite[section 5.4]{Howie}.

\subsection{Munn semigroups}

Let $E$ be a semilattice. For each $x\in E$ let 
\[
Ex=\{y\in E: y\leq x\}
\]
the principal ideal generated by $x$.
\[
\mathcal{U}=\{(x,y)\in E\times E:\; Ex\approx Ey\}.
\]
where $Ex\approx Ey$ means that there is a lattice isomorphism between $Ex$ and $Ey$, that is to say, there is an order isomorphism between them. For each $(x,y)\in E$, we let
\[
T_{x,y}=\{f:\; f:Ex\to Ey\;\;\mbox{is an isomorphism}\}
\]
and 
\[
T(E)=\bigcup\{T_{x,y}:\; (x,y)\in \mathcal{U}\}.
\]
It is known that $T(E)\subseteq I(E)$ and,  in fact, $T(E)$ is an inverse subsemigroup of $I(E)$ called the Munn semigroup associated to $E$.

\begin{prop}
Let $E$ be a semilattice. Then  

\begin{itemize}
\item[(i)] $T(E)$ is a closed  in $I(E)$ iff $\{Ex: x\in E\}$ is closed in $2^E$.

\item[(ii)] In $E$ is countable, then $\{Ex: x\in E\}$ and $T(E)$ are $F_\sigma$.

\end{itemize}

\end{prop}

\begin{proof}
(i) Suppose $T(E)$ is closed. Let $1_{Ex}$ be the identity function on $Ex$. Then $Ex=\dom(1_{Ex})$. Using that $x\mapsto\dom(1_{Ex})$ is a continuous function, we easily get $\{Ex:\; x\in E\}$ is  closed in $2^E$. 
Conversely, suppose $\{Ex: x\in E\}$ is closed in $2^E$. Let $f\in I(E)$. Then  $f\in T(E)$ iff the following conditions hold:
\begin{itemize}
\item[(a)] $\dom(f), \im(f)\in \{Ex:\; x\in E\}$.
    
\item[(b)] For all $x, y, w,z\in E$, if $
    f\in v(x,w)\cap v(y,z)$, then $\;(x\leq y\;\Longleftrightarrow w\leq z)$.
\end{itemize}

Notice that both conditions (a) and (b) define closed subsets of $I(E)$.

(ii) Suppose $E$ is countable. Then clearly $\{Ex: x\in E\}$ is $F_\sigma$.  And from (a) and (b) above we conclude that $T(E)$ is $F_\sigma$.
\end{proof}

\begin{ex}
Let $E$ be the well ordered set $\omega +1$, that is, $E=\N\cup\{x_\infty\}$ where $\N$ its usual order and $x_\infty$ is the maximum of $E$.  Then $En=\{m\in \N: m\leq n\}$ and $Ex_\infty=E$. Clearly $\mathcal{U}=\Delta_E=\{(x,x): x\in E\}$.
Also $T_{x,x}=\{1_{Ex}\}$ for all $x\in E$.
That is to say
\[
T(E)=\{1_{Ex}:\; x\in E\}.
\]
Clearly, $T(E)$ is a closed subset of $I(E)$.
\end{ex}

\subsection{Partial isomorphism on a lattice of open sets}
We present a generalization of the Munn semigroups which provides interesting examples of Polish inverse semigroups. Similar but different inverse subsemigroups of $I(\N)$ were constructed in \cite{APU2022}.

Let $X$ be a compact metric space.  Let $\base$ be a countable basis for $X$. We will always assume that all bases are closed under finite unions and intersection and contains $\emptyset$. We say that $f\in I(\base)$ is an {\em order partial isomorphism} if  for every $u, v\in \dom(f)$,  
\[
f(u)\subseteq f(v) \;\Longleftrightarrow \;u\subseteq v
\]
and $f$ is a {\em complete map} if  for every sequence $u_n$ in $\dom(f)$ we have
\[
\bigcup_n u_n\in \dom(f) \Longleftrightarrow \bigcup_nf(u_n)\in \im(f).
\]
Let 
\[
\mathcal{ISO}(\base)=\{f\in I(\base ):\; \mbox{$f$ is an order isomorphism} \}
\]
and
\[
\mathcal{CISO}(\base)=\{f\in \mathcal{ISO}(\base):\; \mbox{$f$ is a complete map} \}.
\]

\begin{lema}
\label{subsemigroups}
Let $X$ be a compact metric space and $\base$ be a base for $X$. Then, $\mathcal{ISO}(\base)$ and $\mathcal{CISO}(\base)$ are  inverse subsemigroup of $I(\base)$  and $\mathcal{ISO}(\base)$ is closed in $(I(\base),\tau_{pp})$.
\end{lema}

\begin{proof} It is clear that $\mathcal{ISO}(\base)$ and $\mathcal{CISO}(\base)$ are inverse subsemigroups of $I(\base)$. 
To show that $\mathcal{ISO}(\base)$ is  closed in $I(\base)$,  let  $f_n\in \mathcal{ISO}(\base)$ for all $n$ and suppose $f_n\to f$ with $f\in I(\base)$. Let $u,v\in \dom(f)$, then there is $n_0$ such that $u,v\in \dom(f_n)$ and $f_n(u)=f(u)$ and $f_n(v)=f(v)$ for all $n\geq n_0$. From this follows that $f$ is a partial order isomorphism.
\end{proof}

We say that $L\subseteq\base$ is {\em hereditary} if for all $u\in L$ and $v\subseteq u$ with $v\in \base$, we have that $v\in L$. 
Let 
\[
\mathcal{L}=\{L\subseteq \base: \mbox{$L$ is hereditary and closed under finite unions and intersections}\}
\]
and 
\[
S(\base)=\{f\in \mathcal{ISO}(\base):\; \dom(f),\im(f)\in \mathcal{L}\}.
\]
For each open set $V\subseteq X$ (including $\emptyset$), let 
\[
\widetilde{V}=\{ u\in \base:\; u\subseteq V\}
\]
and
\[
\mathcal{C}=\{\widetilde{V}:\; V\subseteq X\; \mbox{is open}\}.
\]
Notice that $\mathcal{C}\subseteq \mathcal{L}$.
Let
\[
\widetilde{S}(\base)=\{f\in \mathcal{CISO}(\base):\; \dom(f),\im(f)\in \mathcal{C}\}.
\]

\begin{lema}
Let $X$ be a compact metric space and $\base$ be a base for $X$. If $f\in S(\base)$, then $f$ preserves finite unions and intersections.
\end{lema}

\begin{lema}
\label{CaracC}
Let $X$ be a compact metric space, $\base$ a base for $X$ and $D\subseteq \base$. Then, $D\in \mathcal{C}$ if and only if    $\bigcup_n u_n\in D$ for all  sequence $(u_n)$ of elements of  $D$ with  $\bigcup_n u_n\in \base$.
\end{lema}
\proof The ``if'' part is clear. For the other direction, suppose $D$ satisfies the condition and let $V=\bigcup D$. We show that $D=\widetilde{V}$. In fact, clearly $D\subseteq \widetilde{V}$. Let $u\in \widehat{V}$. Since $\base$ is a basis, then there are $u_n\in D$ such that $u=\bigcup_n u_n$. By hypothesis $u=\bigcup_n u_n\in D$.
\endproof

\begin{lema}
\label{closedsemigroups}
Let $X$ be a compact metric space and $\base$ be a countable base for $X$. Then, 

\begin{itemize}
\item[(i)] $\mathcal{L}$ is a closed subset of $2^{\base}$.

\item[(ii)] If $f\in S(\base)$ and $L\in \mathcal{L}$, then $f(L\cap \dom(f))\in \mathcal{L}$.
\item[(iii)] $S(\base)$ is a closed inverse subsemigroup  of $\mathcal{ISO}(\base)$ and therefore $S(\base)$ is Polish.

\item[(iv)] $\mathcal{C}$ is an $F_{\sigma\delta}$ subset of $2^\base$

\item[(v)] $\widetilde{S}(\base) $ is an $F_{\sigma\delta}$ inverse subsemigroup of $\mathcal{ISO}(\base)$.

\end{itemize}

\end{lema}

\begin{proof}
(i) $ L$ is not hereditary iff  there are  $u, v\in \base$ such that $u\subseteq v$ with $v\in L$ and $u\not\in L$. This  implies that the collection of non hereditary sets is open.  The rest is analogous. 

(ii) Let $f\in S(\base)$ and $L\in \mathcal{L}$.  Let us see that $f(L\cap \dom(f))$ is hereditary. Assume $v\in f(L\cap \dom(f))$ and $u\in B$ is such that $u\subseteq v$. Since $v\in \im(f)$, $u\subseteq v$ and $\im(f)\in \mathcal{L}$, we have that $u\in \im(f)$. Let $o\in \dom(f)$ be such that $u=f(o)$. We wish to show that $o\in L\cap \dom(f)$.
Since $v\in f(L\cap \dom(f))$, there is $w\in L\cap \dom(f)$ such that $v=f(w)$.  
Then $f(o)\subseteq f(w)$ and, as $f\in \mathcal{ISO}(\base)$, $o\subseteq w$. As $w\in L$ and $L$ is hereditary,  $o\in L$. Therefore $o\in L\cap \dom(f)$.

Now, let us see that $f(L\cap \dom(f))$ is closed under finite unions.
Let $o,w\in f(L\cap \dom(f))$, then there are $u,v\in L\cap \dom(f)$ such that $o=f(u)$ and $w=f(v)$, then $u\cup v \in L\cap \dom(f)$. Since $f\in S(\base)$,  $f(u)\cup f(v)\subset f(u \cup v)$. As  $\im(f)$ is hereditary, $f(u)\cup f(v)\in \im(f)$. Let $q\in \dom(f)$ be such that $f(q)=f(u)\cup f(v)$. Then $f(q)\subseteq f(u\cup v)$, and thus $q\subset u\cup v$. Since $u\cup v\in L$ and $L$ is hereditary,  $q\in L$, that is $f(q)=o\cup w\in f(L\cap \dom(f))$.

(iii) It is clear that $S(\base)$ is closed under the inverse operation. Let us see that $S(\base)$ is a semigroup. Let $f,g\in S(\base)$, since $f,g^{-1}\in S$, $\dom(f),\im(g)\in \mathcal{L}$, we have  by  $(ii)$ that $\dom(f\circ g)=g^{-1}(\im(g)\cap \dom(f))\in  \mathcal{L}$ and $\im(f\circ g)=f(\im(g)\cap \dom(f))\in  \mathcal{L}$.  To verify that $S(\base)$ is topologically closed just observe that $\dom, \im$ are continuous, $\mathcal{L}$ and $\mathcal{ISO}$ are (topologically) closed. 

(iv) Notice that
\[
\begin{array}{lcl}
A\in \mathcal{C} & \Longleftrightarrow &\forall u\in\base\; (\; u\subseteq \bigcup A \rightarrow \; u\in A)\\
 & \Longleftrightarrow &\forall u\in\base\;\{\;\forall x\in X\,[ \, x\in u\rightarrow \exists v\in \base\, (\,v\in A\wedge x\in v)]\; \rightarrow u\in A\,\}.
\end{array}
\]
Let $R_u\subseteq 2^\base\times X$ for $u\in \base$ defined by 
\[
R_u=\{\; (A,x):\;  x\in u\rightarrow \exists v\in \base\, (\,v\in A\wedge x\in v)\;\}.
\]
Then,  $R_u$ is $G_\delta$ (it is a union of a closed set and an open set). As $X$ and $2^\base$ are compact, $\{A\in 2^\base:\; \forall x\in X \;(A,x)\in R_u\}$ is $G_\delta$. Therefore $\mathcal{C}$ is  $F_{\sigma\delta}$. 

(v) By Lemma \ref{subsemigroups}, $\mathcal{CISO}(\base)$ is an inverse semigroup. Thus, to verify that $\widetilde{S}(\base)$ is a semigroup, we only need to show that if $f,g\in \widetilde{S}(\base)$, then $\dom(f\circ g), \im(f\circ g)\in \mathcal{C}$. This follows from Lemma \ref{CaracC}.

To show that $\widetilde{S}(\base) $ is an $F_{\sigma\delta}$ set, consider the following collection:
\[
T=\{f\in \mathcal{CISO}(\base):\; \dom(f), \im(f)\in \mathcal{L}\}=S(\base)\cap \mathcal{CISO}(\base).
\]
We claim that $T$ is closed in $I(\base)$. In fact, suppose $f_n \in T$ and $f_n \to f$. Then by part (iii), $f\in S(\base)$. We need to verify that $f$ is a complete map. Let $u_n \in \dom(f)$ be  such that $u=\bigcup_n u_n \in \dom(f)$. Therefore, there is $n_0$ such that  $u\in \dom(f_n)$ for all $n\geq n_0$. Since $\dom(f_n)\in \mathcal{L}$, $u_m \in \dom(f_n)$ for all $m$ and all $n\geq n_0$. As $f_n$ is a complete map, we conclude that $v=\bigcup_m f_n(u_m)\in \im(f_n)$ for all $n\geq n_0$. Hence  $v\in \im(f)$. The reciprocal is analogous and we conclude that $f$ is a complete map. 

Finally we observe that $f\in \widetilde{S}(\base)$ iff $f\in T$ and $\dom(f), \im(f)\in \mathcal{C}$. Since $T$ is closed and  $\dom$ and $\im$ are continuous, we conclude by  part (iv) that $\widetilde{S}(\base)$ is $F_{\sigma\delta}$. 
\end{proof}

Let 
\[
S_1(\base)=\{f\in S(\base): \forall u,v\in \base\;(\;\overline{v}\subseteq u\in \dom(f)\;\Leftrightarrow\; \overline{f(v)}\subseteq f(u))\}.
\]

\begin{lema}
\label{S1-closed}
Let $X$ be a compact metric space and $\base$ be a countable base for $X$. Then, $S_1(\base)$ is a closed inverse subsemigroup of $S(\base)$ and  hence  Polish. 
\end{lema}

\proof
Let $f,g\in S_1(\base)$ and $u,v\in \base$. Suppose $\overline{v}\subseteq u\in\dom(f\circ g)$. Then $u\in \dom(g) $. As $g\in S_1(\base)$,  $\overline{g(v)}\subseteq g(u)$. As $g(u)\in \dom(f)$ and $f\in S_1(\base)$, $\overline{f(g(u))}\subseteq f(g(u))$. Conversely, if $\overline{f(g(u))}\subseteq f(g(u))$, then $\overline{v}\subseteq u$. Hence $f\circ g\in S_1(\base)$. A similar argument shows that $f^{-1}\in S_1(\base)$. 

It is straightforward to verify that if $f_n\in S_1(\base)$ and $f_n\to f$, then $f\in S_1(\base)$. 
\endproof

\begin{lema}
\label{C=L}
Let $X$ be a 0-dimensional compact metric space and $\base$ be a countable base for $X$ consisting of clopen sets. Then $\mathcal{L}=\mathcal{C}$ and $S(\base) = S_1(\base)$. 

\end{lema}

\begin{proof}
It is clear that $\mathcal{C}\subseteq\mathcal{L}$. Let $L\in \mathcal{L}$, let us see that $L=\widehat{\bigcup L}$. Clearly $L\subseteq \widehat{\bigcup L}$. Let  $u\in \widehat{\bigcup L}$, then $u\subseteq \bigcup L$, then for each $x\in u$, there is $u_x\in L$ such that $x\in u_x$, since $u$ is clopen (thus, compact), there are $x_1,...,x_n\in u$ such that $u\subseteq u_{x_1}\cup...\cup u_{x_n}$, as $L$ is closed under finite unions and it is hereditary, we have that $u\in L$.
\end{proof}

\section{Polishness  of  $\Gamma(X)$}

In this section we present one of the main result of this paper. Let $X$ be a 0-dimensional compact metric space, we show that $\Gamma(X)$ is isomorphic to $S(\base)$ for $\base$ a base of clopen subsets of $X$. The idea is to codify a partial homeomorphism with a partial $\subseteq$-isomorphism between subsets of $\base$. 

Let $X$ be a compact metric space and $\base$ be a countable  base for $X$. The following subsemigroup of $\Gamma(X)$ contains all  partial homeomorphisms that can be coded by an element of $S(\base)$. Consider
\[
\Gamma(X,\base)=\{f\in \Gamma(X):\; f[u],f^{-1}[v]\in \base\;\mbox{for all $u\subseteq\dom(f)$ and $v\subseteq \im(f)$ with $u,v\in \base$}\}
\]
and
\[
H(X,\base)=\{f\in H(X):\; f\in \Gamma(X, \base)\}.
\]

\begin{prop}
\label{GammaX-Borel}
Let $X$ be a compact metric space and $\base$ be a countable base for $X$. Then $\Gamma(X,\base)$ and $H(X,\base)$ are Borel in $\Gamma(X)$.
\end{prop}

\proof
It follows from Proposition \ref{equality} and the following   equivalence
\[
f\in \Gamma(X,\base)\;\Longleftrightarrow \; (\forall U,U'\in \base)\;(\exists V,V'\in \base)  \; (f\in E(U;V)\;\&\; f^{-1}\in E(U'; V') ).
\]
For $H(X,\base)$ just observe that $\dom(f)=X$ iff $f\in \langle \overline{U};X\rangle$ for all $U\in \base$ and the same for $\im(f)$.
\endproof

For $A\subseteq \base$, let 
\[
\widehat{A}=\bigcup\{u:\; u\in A\}.
\]
To each $f\in S_1(\base)$ we associate a  function  $\widehat{f}:\widehat{\dom(f)}\to \widehat{\im(f)}$ as follows:
\begin{equation}
\label{fgorro}
\widehat{f}(x)=y \;\Longleftrightarrow \; y\in f(u) \; \mbox{for all $u\in \dom(f)$ with $x\in u$}.
\end{equation}
The following lemma shows that $\widehat{f}$ is well defined.

\begin{lema} 
\label{fhat}
Let $X$ be a compact metric space and $\base$ be a countable  base for $X$. Then,  $\widehat{f}\in\Gamma(X, \base)$ for all $f\in S_1(\base)$.  Moreover,  $\widehat{f}[u]=f(u)$ for all $u\in \dom(f)$ and  $\widehat{f}\in H(X,\base)$ if $f\in S_\infty(\base)$.
\end{lema}

\proof
 
Let $f\in S_1(\base)$ and $V=\bigcup\{u: u\in \dom(f)\}$ and $W=\bigcup\{w: w\in \im(f)\}$. We will show that $\widehat{f}:V\to W$. 
Fix $x\in u_0\in dom(f)$. We are going to construct a sequence $\{u_k\in dom(f): k\in \mathbb{N}\}$ satisfying the following conditions for each $k\in \mathbb{N}$:

\begin{itemize}
\item $x\in u_k\subseteq  u_0$ and $diam(u_k)\leq 1/(k+1)$,
    
\item $u_{k+1}\subseteq u_k$,
    
\item $\overline{f(u_{k+1})}\subseteq f(u_k)$ with $diam(\overline{f(u_{k+1})})\leq 1/(k+1)$,
\end{itemize}

Suppose that we have built the first $u_1,...,u_k$ and we show how to construct $u_{k+1}$.

Let $v\in \base$ be such that $x\in v\subseteq \overline{v}\subseteq u_k$.  Since $\overline{f(v)}$ is a compact subset of $f(u_k)$, pick  $w_1,\cdots, w_m$ in $\base$ such that $\overline{w_i}\subseteq f(u_k)$,  $diam(w_i)\leq 1/(k+1)$ and $\overline{f(v)}\subseteq w_{1}\cup ... \cup w_{m}$.
Since $f$ respects finite unions, $x\in v\subseteq f^{-1}(w_{1})\cup\cdots \cup f^{-1}(w_{m})$. Then, there is $l$ such that  $x\in f^{-1}(w_l)$. Pick $u_{k+1}\in \base$ such that $x\in u_{k+1}\subseteq f^{-1}(w_l)$ and $diam(u_{k+1})\leq 1/(k+1)$. Notice that $f^{-1}(w_i)\subseteq u_k$, thus $u_{k+1}\subseteq u_k$. Clearly $\overline{f(u_{k+1})}\subseteq \overline{w_i}\subseteq f(u_k)$. Hence $u_{k+1}$ satisfies the required conditions.

By the compactness of $X$ and the previous construction, we have that for each $x\in V$ there is a unique $y\in f(u_0)\subseteq W$ such that 
\begin{equation}
\label{fgorro2}
\{y\}=\displaystyle\bigcap_{k\in \mathbb{N}} f(u_k)=\bigcap_{k\in \mathbb{N}} \overline{f(u_k)}.
\end{equation}
Since the diameter of the $u_k$ and $f(u_k)$ goes to $0$, it is easy to verify that  $\widehat{f}(x)=y$. 
To see that $\widehat{f}$ is injective it suffices to note that if $u,v\in \dom(f)$ are disjoint, then so are $f(u)$ and $f(v)$.

Now we show that $\widehat{f}$ is continuous. Notice that every open subset of $W$ is a union of sets in $\im(f)$, to show the continuity of $\widehat{f}$ it suffices to verify that $(\widehat{f\,})^{-1}(w)={f}^{-1}(w)$ for all $w\in \im(f)$.  To see that $f^{-1}(w)\subseteq (\widehat{f\,})^{-1}(w)$, pick $x\in f^{-1}(w)$. As $f^{-1}(w)\in \base$,  from \eqref{fgorro} we immediately get that $x\in (\widehat{f\,})^{-1}(w)$. 
Conversely, let $x\in (\widehat{f\;})^{-1}(w)$. Then $\widehat{f}(x)\in w$. By the definition of $\widehat{f}$ (and the proof of \eqref{fgorro2}) it is easy to see that there is $u\in \dom(f)$ such that $x\in u$ and $f(u)\subseteq w$.  Since $f\in S_1(\base)$, $u\subseteq f^{-1}(w)$ and thus $x\in f^{-1}(w)$.

Now we show that $\widehat{f}$ is open. 
As before, it suffices to show that  $\widehat{f}[u]= f(u)$ for all $u\in\dom(f)$. Let $u\in\dom(f)$,  it is clear from  \eqref{fgorro} that $\widehat{f}[u]\subseteq f(u)$. Conversely,  let $y\in f(u)$. As in the proof of \eqref{fgorro2}, we see that there is $x\in V$ such that $x\in f^{-1}(w)$ for all $w\in \im(f)$ with $y\in w$. In particular $x\in u$ and $\widehat{f}(x)=y$. This shows that $\widehat{f}$ is open and, moreover,  that $\widehat{f}$ is surjective. Therefore, we have that $\widehat{f}\in \Gamma(X,\base)$. 
\endproof

\begin{lema}
\label{saturacion}
Let $X$ be a compact metric space and $\base_0$ be a countable base for $X$. For each $h\in \Gamma(X)$ there is a base $\base_1\supseteq \base_0$ and $f\in S_1(\base_1)$ such that $\widehat{f}=h$. In particular, if  $D$ is  a countable subset of $\Gamma(X)$, there is a countable  base $\base_1\supseteq \base_0$  such that $D\subseteq \Gamma(X,\base)$.

\end{lema}
\begin{proof} 
Let $\base_1$ be the saturation of $\base_0$ under the function $h$, that is, the smallest collection of open sets $\base$  such that 
\begin{itemize}
    \item[(i)] $\base_0\subseteq \base$.

\item[(ii)] $h[u], h^{-1}[u]\in \base$ for every $u\in \base$. 
\end{itemize}
Let $V=\dom(h)$ and $W=\im(h)$. Let  $f:\widetilde{V}\to \widetilde{W}$ given by $f(u)=h[u]$ for all $u\in\base_1$ with $u\in \widetilde{V}$.  Notice that if $\overline{v}\subseteq u\in\widetilde{V}$, then $h[\overline{v}]=\overline{h[v]}$ as $h$ is a partial homeomorphism. Therefore $\overline{f(v)}=\overline{h[v]}=h_n[\overline{v}]\subseteq h[u]=f(u)$. Thus $f\in S_1(\base_1)$.  Clearly, $\widetilde{f}=h$. 
This construction can be done analogously with a countable collection $(h_n)_n$ of elements of $\Gamma(X)$. 
\end{proof}

\begin{thm}
\label{copy-S1}
Let $X$ be a compact metric space  and $\base$ be a countable  base for $X$. Then,   $\varphi:S_1(\base)\to \Gamma(X,\base)$ given by $\varphi(f)=\widehat{f}$ is a continuous surjective  homomorphism (as inverse semigroups).
\end{thm}

\begin{proof}
To see that $\varphi$ is continuous, let $K$ be   compact subset of $X$, $V$ an open subset of $X$ and $f\in \varphi^{-1}(\left<K;V\right>)$. Since $\widehat{f}$ is continuous and $\widehat{f}[K]\subseteq V$ is compact,  there are open sets $u_1, \cdots, u_n\in \dom(f)$ such that  $K\subseteq u_1\cup\cdots\cup u_n$ and $\widehat{f}[u_i]\subseteq V$ for $1\leq i\leq n$.  Thus, $\displaystyle f\in v(u_{1},f(u_{1}))\cap\cdots\cap v(u_{n},f(u_{n}))\cap S_1\subseteq \varphi^{-1}(\left<K;V\right>)$. Hence,  $\varphi^{-1}(\left<K;V\right>)$ is open. By a symmetrical argument $\varphi^{-1}(\left<K;V\right>^{-1})$ is also open.

Now, let $O$ be open in $X$ and $f\in \varphi^{-1}(D^{-1}(O^-))$. Then, there exists $x\in (X\setminus\dom(\widehat{f}))\cap O$, thus let $u\in \base$ such that $x\in u\subseteq \overline{u}\subseteq O$. Hence,  $u\not\in \dom(f)$, that is,  $f\in w_1(u)$. Let us see that $w_1(u)\cap S_1(\base)\subseteq \varphi^{-1}(D^{-1}(O^-))$. Let $g\in w_1(u)\cap S_1(\base)$. Suppose that $g\not\in \varphi^{-1}(D^{-1}(O^-))$, then $u\subseteq \overline{u}\subseteq O\subseteq \dom(\widehat{g})$. Since $\overline{u}$ is compact, there are $u_1,\cdots, u_n\in \dom(g)$ such that $u\subseteq \overline{u}\subseteq u_{1}\cup\cdots\cup u_{n}$. Since $\dom(g)\in \mathcal{L}$, it is closed under finite unions and thus $u_{1}\cup\cdots\cup u_{n}\in \dom(g)$. Then $u\in \dom(g)$, which is a contradiction.

To check that $\varphi$ is an homomorphism of semigroups,  let $f,g\in S(\base)$. We first  show that $\dom(\widehat{f\circ g})=\dom(\widehat{f}\circ \widehat{g})$. Let $x\in \dom(\widehat{f\circ g})$, then there exists $u\in \dom(f\circ g)$ such that $x\in u$. Then, $x\in u\in \dom(g)$ and $\widehat{g}(x)\in g(u)\in \dom(f)$, that implies that $x\in \dom(\widehat{f}\circ\widehat{g})$. Conversely, let $x\in \dom(\widehat{f}\circ \widehat{g})$, then there exists $u\in \dom(g)$ such that $x\in u$ and $v\in \dom(f)$ such that $\widehat{g}(x)\in v$. Note that $\widehat{g}(x)\in g(u)\cap v\in \dom(f)\cap \im(g)$ (as $\dom(f)$ and $\im(g)$ are hereditary).  Define $w=g^{-1}(g(u)\cap v)$, then we have that $w\in \dom(g)$, $g(w)\in \dom(f)$ and $x\in w$, therefore $x\in \dom(\widehat{f\circ g})$.
    
Now we show that $(\widehat{f\circ g})(x) =(\widehat{f}\circ \widehat{g})(x)$ for all $x\in \dom(\widehat{f\circ g})$. Since $x\in \dom(\widehat{g})$, there exists $u\in \dom(g)$ such that $x\in u$. Also,  there is $v\in \dom(f)$ such that $\widehat{g}(x)\in v$.  Consider $w=g^{-1}(g(u)\cap v)$, then $x\in w\in \dom(f\circ g)$ and by the definition of $\widehat{f\circ g}$, we have that  $(\widehat{f\circ g})(x)\in f(g(w))\subseteq f(v)$. Therefore, $(\widehat{f\circ g})(x)=\widehat{f}(\widehat{g}(x))$.
    
From \eqref{fgorro}, it follows that $\varphi(f^{-1})=(\varphi(f))^{-1}$.  
Finally, it remains to show that $\varphi$ is surjective. Let $h\in \Gamma(X,\base)$, $h:V\to W$. Let $\widetilde{V}=\{u\in \base: \; u\subseteq V\}$ and define $\widetilde{W}$ analogously. Let  $f:\widetilde{V}\longrightarrow \widetilde{W}$ by $f(u)=h[u]$ for each $u\subseteq \dom(h)$ with $u\in \base$.  Since $h$ is an partial homeomorphism,  $f\in S_1(\base)$. From the definition \eqref{fgorro} of $\widehat{f}$ it follows easily that  $\widehat{f}=h$.
\end{proof}

\begin{rem}
In general,  the function $\varphi$ given in Theorem \ref{copy-S1} is not injective. Consider $X=[0,1]$ with the usual basis $\base$ of intervals with rationals endpoints. Let  $V_n=(1/n,1)$ for $n\geq 2$, $A=\bigcup_n \widehat{V}_n$,   $f_1=1_A$, $B=\widehat{(0,1)}$ and  $f_2=1_B$. Observe que $A\neq B$, but $\widehat{f_1}=\widehat{f_2}=1_{(0,1)}$.
\end{rem}

Now we are ready to show one of the main result of the paper.

\begin{thm}
\label{0dim}
Let $X$ be a compact metric space. 
If $X$ is 0-dimensional and $\base$ is a basis for $X$ of clopen sets, then $\Gamma(X)$ is topologically  isomorphic to $S(\base)$ and therefore $\Gamma(X)$ is Polish. 
\end{thm}

\begin{proof}
By Lemma \ref{C=L}, $S_1(\base)=S(\base)$ as the sets in $\base$ are clopen.
By 
Let $\varphi: S(\base)\rightarrow \Gamma(X)$ be given by $\varphi(f)=\widehat{f}$, then $\varphi$ is continuous inverse semigroup homomorphism by Theorem \ref{copy-S1}.

Let us see that $\varphi$ is injective. Let $f,g\in S(\base)$ with $f\neq g$. There are two cases to be considered: (a) there is $u\in\dom(f)\setminus\dom(g)$.
Then, as $u$ is clopen and $\dom(g)\in \mathcal{L}$,  $u\not\subseteq \dom(\widehat{g})$ and therefore $\widehat{f}\neq \widehat{g}$. (b) $\dom(f)=\dom(g)$ and there is $u\in\dom(f)$ such that $f(u)\neq g(u)$. Then, clearly   $\widehat{f}\neq \widehat{g}$.

Let us see that $\varphi$ is surjective. Let $h\in \Gamma(X)$, $h:V\to W$. Let $\widetilde{V}=\{u\in \base: \; u\subseteq V\}$ and define $\widetilde{W}$ analogously. Let  $f:\widetilde{V}\longrightarrow \widetilde{W}$ by $f(u)=h[u]$ for each clopen $u\subseteq \dom(h)$. It is easy to see that $f\in S(\base)$ and $\widehat{f}=h$.
    
Finally, let us see that $\varphi$ is open. Let  $o,p\in \mathcal{B}$, as they are clopen sets, the basic open set   $\left<o;p\right>$ is well defined. 
Let $f\in \mathcal{S}(\base)$, thus $f\in \mathcal{S}_1 (\base)$. Let $V$ be open such that $\dom(f)=\widetilde{V}$. Then (i) $f\in v(o;p)$ iff $\widehat{f}[o]=p$. (ii) $f\in w_1(o)$ iff $o\not\subseteq V=\dom(\widehat{f})$ iff ${\sf D}(\widehat{f})\in o^{-}$ and analogously (iii) $f\in w_2(o)$ iff ${\sf I}(\widehat{f})\in o^{-}$.
From these facts and the surjectivity of $\varphi$, we have that:
\begin{itemize}
\item $\varphi(v(o,p)\cap S(\base))=\left<o;p\right>\cap \left<p;o\right>^{-1} $.
    
\item $\varphi( w_1(o)\cap S(\base))=\DD^{-1}(o^{-})$.
    
\item $\varphi(w_2(o)\cap S(\base))=\II^{-1}(o^{-})$.
\end{itemize}
\end{proof}

\begin{rem}
The previous theorem is an equivalence:
Consider $\psi:X\to \Gamma(X)$ given by $\psi(x)=1_{X\setminus \{x\}}$. Then $\psi $ an embedding. Thus, if $\Gamma(X)$ is isomorphic to a subsemigroup of $I( \N)$, then $X$ is necessarily 0-dimensional.
\end{rem}

\subsection{Some subsemigroups of $\Gamma(X)$}

Consider
\[
\Gamma_d(X)=\{f\in \Gamma(X):\; \dom(f), \im(f)\;\mbox{are dense sets}\}.
\]

\begin{thm}
Let $X$ be a compact metric space.  Then, $\Gamma_d(X)$ is a $G_\delta$ inverse subsemigroup of $\Gamma(X)$. In particular, if $X$ is 0-dimensional, $\Gamma_d(X)$ is a Polish inverse semigroup.
\end{thm}

\begin{proof}
It is known and easy to verify  that $NWD=\{K\in CL(X):\; K\; \mbox{is nowhere dense}\}$ is $G_\delta$. Now observe that 
$\Gamma_d(X)= {\sf D}^{-1}(NWD)\;\cap \;{\sf I}^{-1}(NWD)$. The last claim follows from Theorem \ref{0dim} and the classical result that every $G_\delta$ subset of a Polish space is also Polish (see, for instance, \cite[Theorem 3.11]{kechris1995}).
\end{proof}

We recall that a measurable space $(Y, \mathcal{A})$, where $\mathcal{A}$ is a $\sigma$-algebra on $Y$ is called {\em standard Borel space} if $(Y, \mathcal{A})$  is isomorphic to a Polish space with its  $\sigma$-algebra of Borel sets (see \cite[\S 12]{kechris1995}). A metric space $Y$ is called standard Borel if $Y$ together with its $\sigma$-algebra of Borel sets is standard Borel.
If $Y$ is a Polish space and $Z\subseteq Y$ is a Borel set, then $Z$, as a metric space, is standard Borel.  We know that $\Gamma(X,\base)$ is a Borel subset of $\Gamma(X)$, however, we do not know if $\Gamma(X)$ is Polish. Nevertheless, we have the following fact. 

\begin{thm}
Let $X$ be a compact metric space and $\base$ be a countable base for $X$. Then,  $\Gamma(X,\base)$ is  standard Borel.
\end{thm}

\proof 
Let $\widetilde{S}_1(\base) =S_1(\base)\cap \widetilde{S}(\base)$.  From Lemma \ref{closedsemigroups} and \ref{S1-closed},  $\widetilde{S}_1(\base)$ is a $F_{\sigma\delta}$ subset of $S(\base)$. We claim that  $\varphi({S}_1(\base))=\varphi(\widetilde{S}_1(\base))$.  In fact, let $f\in S_1$ and  $V=\bigcup \dom(f)$ and $W=\bigcup \im(f)$. Let $f_1:\widetilde{V}\to \widetilde{W}$ given by $f_1(u)=\widehat{f}[u]$. Then $\varphi(f)=\varphi(f_1)$.

Let $\psi: \widetilde{S}_1(\base)\to \Gamma(X,\base)$ be the restriction of $\varphi$ to $\widetilde{S}_1(\base)$. Then $\psi$  is a continuous bijection (by the proof of Theorem \ref{copy-S1}). By Lemma \ref{GammaX-Borel}, $\Gamma(X,\base)$ is a Borel set  in $\Gamma(X)$. So, it remains to verify that $\psi$  is a Borel isomorphism, that is, that it sends basic open sets  into Borel sets.  As in the proof of Theorem \ref{0dim}, let $o,p\in\base$. Recall the sets $E(o;p)=\{f\in\Gamma(X): \; f[u]=p\}$ defined in Proposition \ref{equality}. Then

\begin{itemize}
\item $\varphi(v(o,p)\cap \widetilde{S}_1(\base)) = E(o;p)\cap \Gamma(X,\base)$.
    
\item $\varphi(w_1(o)\cap \widetilde{S}_1(\base))=\DD^{-1}(o^{-})\cap \Gamma(X,\base)$.
    
\item $\varphi( w_2(o)\cap \widetilde{S}_1(\base))=\II^{-1}(o^{-})\cap \Gamma(X,\base)$.
\end{itemize}
By Proposition \ref{equality},  $E(o;p)$ is Borel in $\Gamma(X)$ and thus we are done. 
\endproof

The main question we have left open is the following.

\begin{question}
Is $(\Gamma(X), \tau_{hco})$ Polish for every compact metric space $X$?
\end{question}

\noindent{\bf Acknowledgment.} We thank Edwar Ram\'irez  for some discussion about metrics on $\Gamma(X)$.

\bibliographystyle{plain}

\begin{thebibliography}{9}

\bibitem{APU2022}
Arana, K., P\'erez, J and Uzc\'ategui, C.
\newblock {Pettis property for Polish inverse semigroups}, accepted to be published by Applied General Topology, 2022.

\bibitem{BeckerKechris1996}
H.~Becker and A.~Kechris.
\newblock {\em The descriptive set theory of {P}olish group actions}, volume
  232 of {\em London Mathematical Society Lecture Note Series}.
\newblock Cambridge University Press, Cambridge, 1996.

\bibitem{Beer1993}
G.~Beer.
\newblock {Topologies on closed and closed convex sets}.
\newblock {\em Kluwer Academic Publisher, Dordercht}, 1993.




\bibitem{elliott2020}
Elliott, L., Jonušas, J.,  Mesyan, Z.,  Mitchell, J. D., Morayne, M. and Péresse, Y.
\newblock Automatic continuity, unique {P}olish topologies, and {Z}ariski topologies on  monoids and clones.
\newblock {arxiv.org/abs/1912.07029v4}
\newblock 2021.

\bibitem{Howie} 
Howie, J. 
\newblock {Fundamentals of semigroup theory},  {London  Mathematical Society Monographs. New Series} 12.
\newblock Oxford University Press, New York, Second
  Edition, 2003.



\bibitem{kechris1995} Kechris, A.  Classical Descriptive Set Theory. Graduate Texts in Mathematics 156. Springer-Verlag, New York, 1995.

\bibitem{Law} Lawson, M. V. “Inverse Semigroups, The Theory of Partial Symmetries,” Word Scientific, Singapore, 1998.


\bibitem{MPU} Mart\'inez, L., Pinedo, H. and Uzc\'ategui, C. A topological correspondence between   partial actions of groups and inverse semigroup actions. {\em Forum Math.}, 34(2), 431-446, 2022.

\bibitem{PerezUzca2022} P\'erez, J. and Uzc\'ategui, C. 
\newblock Topologies on the symmetric inverse semigroup.
\newblock {\em Semigroup Forum}, 104(2), 398-414, 2022.


\end{thebibliography}

\end{document}